\documentclass[12pt]{article}%

\usepackage{amsmath,enumerate}
\usepackage{amsfonts}
\usepackage{amssymb}
\usepackage{color}
\usepackage{mathtools}

\newtheorem{theorem}{Theorem}[section]

\newtheorem{lemma}[theorem]{Lemma}

\newtheorem{proposition}[theorem]{Proposition}

\newenvironment{proof}[1][Proof]{\noindent\textbf{#1.} }
{\hfill \ \rule{0.5em}{0.5em}}

\begin{document}

\title{Degree Ramsey numbers for even cycles}

\date{}

\author{Michael Tait\thanks{Department of Mathematical Sciences, Carnegie Mellon University. \texttt{mtait@cmu.edu}. Research supported by NSF grant DMS-1606350.}}

\maketitle

\abstract{Let $H\xrightarrow{s} G$ denote that any $s$-coloring of $E(H)$ contains a monochromatic $G$. The degree Ramsey number of a graph $G$, denoted by $R_\Delta(G, s)$, is $\min \{\Delta(H): H \xrightarrow{s} G \}$. We consider degree Ramsey numbers where $G$ is a fixed even cycle. Kinnersley, Milans, and West showed that $R_\Delta(C_{2k},s) \geq 2s$, and Kang and Perarnau showed that $R_\Delta(C_4, s) = \Theta(s^2)$. Our main result is that $R_\Delta(C_6, s) = \Theta(s^{3/2})$ and $R_\Delta(C_{10}, s) = \Theta(s^{5/4})$. Additionally, we substantially improve the lower bound for $R_\Delta(C_{2k}, s)$ for general $k$.}

\section{Introduction}

Theorems in Ramsey theory state that if a structure is in some suitable sense ``large enough", then it must contain a fixed substructure. The classical {\em Ramsey number} of a graph $G$, denoted by $R(G,s)$, is the smallest $n$ such that $K_n \xrightarrow{s} G$, where $H\xrightarrow{s} G$ denotes that any $s$-coloring of the edges of $H$ produces a monochromatic subgraph isomorphic to $G$. Classical Ramsey numbers may be thought of in a more general setting, as just one type of parameter Ramsey number. Note that the classical Ramsey number of $G$ is $\min \{ |V(H)|: H\xrightarrow{s} G\}$. For any monotone graph parameter $\rho$, one may define the {\em $\rho$-Ramsey number} of $G$, denoted by $R_\rho(G,s)$, to be
\[
\min\{ \rho(H): H\xrightarrow{s} G\}.
\]

This generalizes the classical Ramsey number as $R_\rho(G,s)$ is the Ramsey number for $G$ when $\rho(H)$ denotes the number of vertices in $H$. The study of parameter Ramsey numbers dates back to the 1970s \cite{BurrErdosLovasz1976}. Since then, many researchers have studied this quantity when $\rho(H)$ is the clique number of $H$ \cite{Folkman1970, LinLi2015, NesetrilRodl1976} (giving way to the study of Folkman numbers), when $\rho(H)$ is the number of edges in $H$ \cite{Beck1983, Beck1990, Dellamonica2012, DonadelliHaxellKohayakawa2005, ErdosEtAl1978, FriedmanPippenger1987, RodlSzemeredi2000} (now called the {\em size Ramsey number}), when $\rho(H) = \chi(H)$ \cite{BurrErdosLovasz1976, Zhu1998, Zhu2011} or when it is the circular chromatic number \cite{jao2013}. In this note we are interested in the {\em degree Ramsey number}, which is when $\rho(H) = \Delta(H)$.

Burr, Erd\H{o}s, and Lov\'asz \cite{BurrErdosLovasz1976} showed that $R_\Delta(K_n, s) = R(K_n, s) - 1$. Kinnersley, Milans, and West \cite{KinnersleyMilansWest2012} and Jiang, Milans, and West \cite{JiangMilansWest2013} proved several theorems regarding the degree Ramsey numbers of trees and cycles. Horn, Milans, and R\"odl showed that the family of closed blowups of trees is $R_\Delta$-bounded (that their degree Ramsey number is bounded by a function of the maximum degree of the graph and $s$). The main open question in this area is whether the set of all graphs is $R_\Delta$-bounded (see \cite{ConlonFoxSudakov2015}).

The main result of this note is to determine the order of magnitude of $R_\Delta(C_6,s)$ and $R_\Delta(C_{10}, s)$. Kang and Perarnau \cite{KangPerarnau2015} showed that $R_\Delta(C_4,s) = \Theta(s^2)$, and for general $k$, the best lower bound on $R_\Delta(C_{2k}, s)$ is by Kinnersley, Milans, and West \cite{KinnersleyMilansWest2012} who show $R_\Delta(C_{2k}, s) \geq 2s$. We substantially improve this lower bound in Theorem \ref{general cycle}. 

As the determination of Ramsey numbers for $C_{2k}$ is closely related to the Tur\'an number for $C_{2k}$, one may find it natural that the order of magnitude for $R_\Delta(C_{2k}, s)$ should be able to be determined for $k \in \{2,3,5\}$ but in no other cases. This is also the current state of affairs for the Tur\'an numbers $\mathrm{ex}(n, C_{2k})$ as well as for the classical Ramsey numbers, where Li and Lih \cite{LiLih2009} showed that $R(C_{2k}, s) = \Theta\left(s^{k/(k-1)}\right)$ for $k\in \{2,3,5\}$. 

Before we state our theorems, we need some preliminary definitions. For graphs $F$ and $G$, a {\em locally injective homomorphism} from $F$ to $G$ is a graph homomorphism $\phi: V(F) \to V(G)$ such that for every $v\in V(F)$, the restriction of $\phi$ to the neighborhood of $v$ is injective. Let $\mathcal{L}_F$ denote the family of all graphs $H$ such that there is a locally injective homomorphism from $F$ to $H$. We say that a graph is $\mathcal{L}_F$-free if it does not contain any $H \in \mathcal{L}_F$. We now state our main theorem.

\begin{theorem}\label{main theorem}
$R_\Delta(C_6, s) = \Theta(s^{3/2})$ and $R_\Delta(C_{10}, s) = \Theta(s^{5/4})$.
\end{theorem}

To prove this theorem, we first show that the complete graph can be partitioned ``efficiently" into graphs coming from the generalized quadrangle and generalized hexagon. Then we use the following general theorem, which is implicit in the work of Kang and Perarnau \cite{KangPerarnau2015}. 

\begin{theorem}[Kang-Perarnau \cite{KangPerarnau2015}]\label{ross theorem}
Let $F$ be a graph with at least one cycle and $\epsilon > 0$ be fixed and let $G$ be a graph of maximum degree $\Delta$. If the edges of $K_n$ can be partitioned into $O\left(n^{1-\epsilon}\right)$ $\mathcal{L}_F$-free graphs, then $G$ can be partitioned into $O\left(\Delta^{1-\epsilon} \right)$ graphs which are $F$-free.
\end{theorem}

They did not state their theorem in this way, and for completeness we sketch its proof in Section \ref{ross section}. Stating it in this general way allows us to improve the result of Kinnersley, Milans, and West \cite{KinnersleyMilansWest2012}. 

\begin{theorem}\label{general cycle}
Let $k$ be fixed and $\delta = 0$ if $k$ is odd and $\delta=1$ if $k$ is even. Then
\[
R_\Delta(C_{2k}, s) = \Omega\left(\left(\frac{s}{\log s}\right)^{1+\frac{2}{3k-5+\delta}}\right).
\]
\end{theorem}

We prove our main theorem in Section \ref{main section}. We sketch the proof of Theorem \ref{ross theorem} and use it to prove Theorem \ref{general cycle} in Section \ref{ross section}. 

\section{Proof of Theorem \ref{main theorem}}\label{main section}
The theorem follows from Theorem \ref{ross theorem} and the following proposition which we prove after the proof of Theorem \ref{main theorem}.

\begin{proposition}\label{complete graph partition}
The edge set of $K_n$ may be partitioned into $O(n^{2/3})$ graphs of girth $8$ or $O(n^{4/5})$ graphs of girth $12$.
\end{proposition}

\begin{proof}[Proof of Theorem \ref{main theorem}]
Showing that $R_\Delta(G, s) \geq k$ is equivalent to showing that any graph of maximum degree at most $k$ may be partitioned into $s$ graphs each of which are $G$-free. We notice that if there is a locally injective homomorphism from $C_n$ to a graph $H$, then $H$ must contain a cycle of length at most $n$. Therefore, if a graph has girth $g$, then it is $\mathcal{L}_{C_{n}}$-free for any $n\leq g$.

Therefore, by Proposition \ref{complete graph partition} and Theorem \ref{ross theorem}, if $G$ is a graph of maximum degree $\Delta$, then $G$ can be partitioned into $O(\Delta^{2/3})$ graphs which are $C_6$-free or $O(\Delta^{4/5})$ graphs which are $C_{10}$-free. This shows that $R_\Delta(C_6, s) = \Omega(s^{3/2})$ and $R_\Delta(C_{10}, s) = \Omega(s^{5/4})$.

The upper bound follows from the classical even-cycle theorem of Erd\H{o}s, that $\mathrm{ex}(n, C_{2k}) = O(n^{1+1/k})$ (cf \cite{BondySimonovits1974}). If $E(K_n)$ is colored with $s$ colors, then one color class contains at least $\binom{n}{2}/s$ edges. Therefore, for a constant $c_k$ depending only on $k$, if $c_k n^{1-1/k} > s$, then any $s$-coloring of $K_n$ contains a monochromatic $C_{2k}$. This implies 
\[
R_\Delta(C_{2k}, s) \leq \left(\frac{s}{c_k}\right)^{1 + \frac{1}{k-1}} - 1.
\]

We note that the best constant $c_k$ that is known comes from the current best-known upper bound for $\mathrm{ex}(n, C_{2k})$ by Bukh and Jiang \cite{BukhJiang2014}.
\end{proof}

\medskip

\begin{proof}[Proof of Proposition \ref{complete graph partition}]
We use the generalized quadrangle and the generalized hexagon to partition $K_n$ efficiently into graphs of girth $8$ and $12$ respectively. Let $q\geq 5$ be a prime and $\mathbb{F}_q$ the field of order $q$. We define bipartite graphs $Q$ and $H$. Let $V(Q) = \mathcal{P}_Q \cup \mathcal{L}_Q$ and $V(H) = \mathcal{P}_H \cup \mathcal{L}_H$ where $\mathcal{P}_Q = \{(p_1, p_2, p_3): p_i\in \mathbb{F}_q\}$, $\mathcal{L}_Q = \{(l_1,l_2, l_3): l_i\in \mathbb{F}_q\}$, $\mathcal{P}_H = \{(p_1, p_2, p_3,p_4, p_5): p_i\in \mathbb{F}_q\}$, and $\mathcal{L}_H = \{(l_1,l_2, l_3, l_4, l_5): l_i\in \mathbb{F}_q\}$. Now define $E(Q)$ by $(p_1,p_2,p_3)\sim (l_1,l_2,l_3)$ if and only if 
\begin{align*}
l_2 - p_2 &= l_1p_1\\
l_3 - 2p_3 &= -2l_1p_2
\end{align*}
and $E(H)$ by $(p_1,p_2,p_3,p_4,p_5)\sim (l_1,l_2,l_3,l_4,l_5)$ if and only if 
\begin{align*}
 l_2 - p_2 &= l_1p_1\\
 l_3 - 2p_3 &= -2l_1p_2\\
 l_4 - 3p_4 &= -3l_1p_3 \\
 2l_5 - 3p_5 &= 3l_3p_2 - 3l_2p_3 + l_4p_1
\end{align*}
The graphs $Q$ and $H$ are $q$-regular bipartite graphs on $2q^3$ and $2q^5$ vertices respectively. In \cite{LazebnikUstimenko1993} (Theorems 2.1 and 2.5), Lazebnik and Ustimenko showed that $Q$ has girth $8$ and $H$ has girth $12$. $Q$ and $H$ are large induced subgraphs of the incidence graph of the generalized quadrangle and generalized hexagon (see also \cite{LazebnikSun2016}). We first show that we may partition $K_{p^3,p^3}$ with disjoint copies of $Q$ and $K_{p^5, p^5}$ with disjoint copies of $H$. 

Let $(\alpha_2, \alpha_3)$ be an arbitrary pair in $\mathbb{F}_q^2$ and $(\beta_2,\beta_3,\beta_4, \beta_5)$ an arbitrary quadruple in $\mathbb{F}_q^4$. Define the graph $Q_{\alpha_2, \alpha_3}$ to be the graph with vertex set $V(Q)$ and $(p_1,p_2,p_3)\sim (l_1,l_2,l_3)$ if and only if 

\begin{align*}
l_2 - (p_2 + \alpha_2) &= l_1p_1\\
l_3 - 2(p_3 + \alpha_3) &= -2l_1(p_2 + \alpha_2)
\end{align*}
and $H_{\beta_2, \beta_3, \beta_4, \beta_5}$ to be the graph with vertex set $V(H)$ and $(p_1,p_2,p_3,p_4, p_5)\sim (l_1,l_2,l_3,l_4,l_5)$ if and only if 
\begin{align*}
 l_2 - (p_2 + \beta_2) &= l_1p_1\\
 l_3 - 2(p_3 + \beta_3) &= -2l_1(p_2 + \beta_2)\\
 l_4 - 3(p_4+\beta_4) &= -3l_1(p_3 + \beta_3)\\
 2l_5 - 3(p_5 +\beta_5) &= 3l_3(p_2 + \beta_2) - 3l_2(p_3 + \beta_3) + l_4p_1
\end{align*}
 Then $Q_{\alpha_2,\alpha_3}$ is isomorphic to $Q$ with the explicit isomorphism given by $(p_1, p_2, p_3) \mapsto (p_1, p_2 - \alpha_2, p_3 - \alpha_3)$ and $(l_1,l_2,l_3) \mapsto (l_1,l_2,l_3)$ and $H_{\beta_2,\beta_3,\beta_4,\beta_5}$ is isomorphic to $H$ with isomorphism given by $(p_1, p_2, p_3, p_4, p_5) \mapsto (p_1, p_2 - \beta_2, p_3 - \beta_3, p_4 - \beta_4, p_5 - \beta_5)$ and $(l_1,l_2,l_3, l_4, l_5) \mapsto (l_1,l_2,l_3, l_4, l_5)$. 
 
 Now we claim that the family $\{Q_{\alpha_2,\alpha_3}\}_{\alpha_i\in \mathbb{F}_q}$ covers $K_{q^3, q^3}$ and $\{H_{\beta_2,\beta_3,\beta_4,\beta_5}\}_{\beta_i \in \mathbb{F}_q}$ covers $K_{q^5, q^5}$. Since each graph is $q$ regular and there are $q^2$ and $q^4$ of them respectively, each cover is also a partition. To show that the edges of the complete bipartite graph are all covered, let $(p_1,p_2,p_3)$ and $(l_1, l_2, l_3)$ be arbitrary and fixed. We must show that there is a choice of $\alpha_2$ and $\alpha_3$ such that 
 
\begin{align*}
l_2 - (p_2 + \alpha_2) &= l_1p_1\\
l_3 - 2(p_3 + \alpha_3) &= -2l_1(p_2 + \alpha_2)
\end{align*}

It is clear that there is a unique solution $\alpha_2, \alpha_3$ to this triangular system. Similarly, for $(p_1, p_2, p_3, p_4, p_5)$ and $(l_1,l_2,l_3,l_4,l_5)$ arbitrary and fixed, there is a unique solution $\beta_2, \beta_3, \beta_4, \beta_5$ to the system 

\begin{align*}
 l_2 - (p_2 + \beta_2) &= l_1p_1\\
 l_3 - 2(p_3 + \beta_3) &= -2l_1(p_2 + \beta_2)\\
 l_4 - 3(p_4+\beta_4) &= -3l_1(p_3 + \beta_3)\\
 2l_5 - 3(p_5 +\beta_5) &= 3l_3(p_2 + \beta_2) - 3l_2(p_3 + \beta_3) + l_4p_1
\end{align*}

Therefore, $K_{q^3, q^3}$ may be covered by $q^2$ graphs each of girth $8$, and $K_{q^5, q^5}$ may be covered with $q^4$ graphs each of girth $12$. Now we must show that we can use a partition of $K_{n,n}$ to find an efficient covering of $K_n$ where we only lose a constant multiplicative factor in the number of graphs used. Li and Lih showed this in \cite{LiLih2009}, and we include the details for completeness. 

By standard results on density of the primes, we have shown above that we may partition $K_{n,n}$ into $(1+o(1))n^{2/3}$ graphs of girth $8$ or $(1+o(1))n^{4/5}$ graphs of girth $12$. Break the vertex set of $K_n$ into two parts $V_1$ and $V_2$ with sizes as equal as possible. We may cover the edges between $V_1$ and $V_2$ with $(1+o(1))(n/2)^{2/3}$ graphs of girth $8$ or $(1+o(1))(n/2)^{4/5}$ graphs of girth $12$. Now break $V_1$ and $V_2$ into equal size pieces $U_1, U_2, U_3, U_4$ each of size $n/4$. Since the disjoint union of two graphs of girth $g$ still has girth $g$, we may cover the edges between $U_1$ and $U_2$ and the edges between $U_3$ and $U_4$ with $(1+o(1))(n/4)^{2/3}$ graphs of girth $8$ or $(1+o(1)(n/4)^{4/5}$ graphs of girth $12$. Repeating this procedure allows us to cover all of the edges in $K_n$ with 
\[
\sum_{i=1}^{O(\log n)} (1+o(1))\left(\frac{n}{2^i}\right)^{2/3} = O(n^{2/3})
\]
graphs of girth $8$ or 
\[
\sum_{i=1}^{O(\log n)} (1+o(1))\left(\frac{n}{2^i}\right)^{4/5} = O(n^{4/5})
\]
graphs of girth $12$.

\end{proof}

We note that this section shows that $R(\{C_3, \cdots C_{2k}\}, s) = \Theta(s^{k/(k-1)})$ for $k\in \{2,3,5\}$ and implies the main result of \cite{LiLih2009}. We also note that we have seen in this section that by Theorem \ref{ross theorem}, giving good lower bounds on degree Ramsey numbers for a graph $F$ can be reduced to giving good lower bounds on the classical Ramsey number for $\mathcal{L}_F$. In the case that $F = K_{a,b}$, we have that $\mathcal{L}_F = \{K_{a,b}\}$. Using the projective norm graphs, Alon, R\'onyai, and Szab\'o \cite{AlonRonyaiSzabo1999} showed that for $a> (b-1)!$, $R(K_{a,b}, s) = \Theta(s^b)$. Along with Theorem \ref{ross theorem}, this shows that for $a>(b-1)!$, one also has $R_\Delta(K_{a,b},s) = \Theta(s^b)$.

\section{Proof of Theorem \ref{ross theorem}}\label{ross section}
Throughout this section, assume that $F$ is a graph with at least one cycle, that $\epsilon > 0$ is fixed, and that the edges of $K_n$ can be partitioned into $O(n^{1-\epsilon})$ graphs which are $\mathcal{L}_F$-free. Call a coloring of a graph a {\em proper rainbow coloring} if the coloring is proper, and the restriction of the coloring to any neighborhood is an injection (ie each vertex sees a rainbow). To prove Theorem \ref{ross theorem} we need the following lemma, which appears in \cite{KangPerarnau2015}. Similar lemmas appear in \cite{PerarnauReed2014} and \cite{MolloySudakovVerstraete2016}.

\medskip
\begin{proof}[Sketch of Proof of Theorem \ref{ross theorem}]
\begin{lemma}\label{ross lemma}
Let $G$ be a graph of sufficiently large maximum degree $\Delta$ and minimum degree $\delta \geq \log^2 \Delta$. Then there is a spanning subgraph $H$ of $G$ and a proper rainbow coloring $\chi(H)$ using at most $200\Delta$ colors such that for all $v\in V(G)$, $d_H(v) \geq \frac{1}{10} d_G(v)$.
\end{lemma}

This lemma allows us to use the partition of $K_n$ to partition a large piece of our graph $G$ (viz $H$) into $F$-free subgraphs.

\begin{proposition}\label{large piece}
Let $G$ be a graph of sufficiently large maximum degree $\Delta$ and minimum degree $\delta \geq \log^2 \Delta$. There exist $l = O(\Delta^{1-\epsilon})$ disjoint spanning subgraphs $H_1, \cdots, H_l$, all of which are $F$-free, such that for all $v\in V(G)$
\[
\sum_{i=1}^l d_{H_i}(v) \geq \frac{1}{10} d_G(v).
\]
\end{proposition}

\begin{proof}
Recall that the edge set of $K_{200\Delta}$ can be partitioned into $l = O(\Delta^{1-\epsilon})$ graphs which are $\mathcal{L}_F$-free. Denote these graphs by $G_1,\cdots, G_l$. Let $H$ be the spanning subgraph of $G$ with coloring $\chi$ from Lemma \ref{ross lemma}. Recall that $\chi$ is a proper rainbow coloring using at most $200\Delta$ colors and that $d_H(v) \geq \frac{1}{10}d_G(v)$ for all $v$. For $1\leq i\leq l$, define graphs $H_i$ which are subgraphs of $H$ by $V(H_i) = V(G)$ and $uv\in E(H_i)$ if and only if
\[
\mbox{$\chi(u)\chi(v) \in E(G_i)$ and $uv\in E(H)$}.
\]

Since $G_1,\cdots, G_l$ is a partition of $E(K_{200\Delta})$, we have that $H_1,\cdots, H_l$  is a partition of $H$, and thus the minimum degree condition is satisfies. To see that each $H_i$ is $F$-free, we claim that for each $i$, $\chi$ is a locally injective homomorphism from $H_i$ to $G_i$. To see this, note that the definition of $E(H_i)$ guarantees that $\chi$ is a homomorphism from $H_i$ to $G_i$, and $\chi$ being a rainbow coloring implies that the homomorphism is locally injective. Since $G_i$ is $\mathcal{L}_F$-free, we have that $H_i$ is $F$-free.
\end{proof}

Let $G_1$ be the graph obtained from $G$ by sequentially removing any vertex of degree less than $\log^2\Delta$, and let $G_2$ be the graph whose edges are $E(G)\setminus E(G_1)$. Since $G_2$ has degeneracy at most $\log^2\Delta$, it has arboricity at most $\log^2 \Delta$ and so we may partition $E(G_2)$ into that many forests. Since $F$ contains a cycle, each of these are $F$-free. Now we may apply Proposition \ref{large piece} to $G_1$, and have therefore covered a large piece of $G$ with at most $O(\Delta^{1-\epsilon}) + \log^2 \Delta$ graphs which are $F$-free. Removing these edges decreases the maximum degree by a multiplicative factor of at least $\frac{9}{10}$. We may repeat this procedure until the maximum degree of part of the graph which is not yet covered is less than $\log^2 \Delta$. Since the maximum degree decreases by a constant multiplicative factor at each step, this will take at most $O(\log \Delta)$ steps. Therefore, the total number of graphs used to cover $E(G)$ is 
\[
\sum_{i=0}^{O(\log \Delta)} 200\left(\left(\frac{9}{10}\right)^i \Delta\right)^{1-\epsilon} + \log^2\Delta = O\left(\Delta^{1-\epsilon}\right) + O(\log^3 \Delta).
\]

\end{proof}

\begin{proof}[Proof of Theorem \ref{general cycle}]
Let $k$ be fixed and let $\delta = 0$ if $k$ is odd and $1$ if $k$ is even. Showing that 
\[
R_\Delta(C_{2k},s) = \Omega\left( \left(\frac{s}{\log s}\right)^{1+ \frac{2}{3k - 5 + \delta}}\right)
\]
is equivalent to showing that any graph of maximum degree $\Delta$ can be partitioned into $O(\Delta^{1- \frac{2}{3k-3+\delta}} \log \Delta)$ graphs each with no copy of $C_{2k}$. By the same argument in the proof of Theorem \ref{main theorem} and by Theorem \ref{ross theorem}, it suffices to show that $K_n$ can be partitioned into $O(n^{1- \frac{2}{3k-3+\delta}} \log n)$ graphs of girth greater than $2k$. Lazebnik, Ustimenko, and Woldar \cite{LazebnikUstimenkoWoldar1995} showed that there are graphs on $n$ vertices and $\epsilon_kn^{1+\frac{2}{3k-3+\delta}}$ edges that have girth at least $2k+2$, where $\epsilon_k$ is a constant depending only on $k$. For $C$ a constant to be chosen later, place $Cn^{1- \frac{2}{3k-3+\delta}} \log n$ copies of this graph onto $K_n$, each time permuting the vertices with a permutation $\sigma \in S_n$ chosen uniformly at random and independently. For each pair $u,v$, let $X_{u,v}$ be the random variable that counts how many times the edge $uv$ in $K_n$ is covered. We are done if we can show that all of the $X_{uv}$ are positive with positive probability. Since the expected value of each $X_{uv}$ is greater than $\epsilon_kC\log n$, the Chernoff bound (cf \cite{AlonSpencer2000}) gives that 
\[
\mathbb{P}(X_{uv} = 0) \leq \mathrm{exp}\left( \frac{ - \epsilon_k C\log n}{3}\right).
\]
For $C$ large enough this is $o(n^{-2})$, and the union bound gives that every edge is covered with probability tending to $1$.
\end{proof}

Indeed, this proof shows a more general theorem:

\begin{theorem}
Let $F$ be a fixed graph containing at least one cycle and let $\mathrm{ex}(n, \mathcal{L}_F) = \Omega( n^{1+\eta})$. Then if $\eta < 1$,
\[
R_\Delta(F, s) = \Omega \left( \frac{s}{\log s}\right)^{\frac{1}{1-\eta}},
\]
and if $\eta= 1$, then 
\[
R_\Delta(F,s) = 2^{\Omega(s^{1/4})}.
\]
\end{theorem}

\section{Concluding Remarks}
Determining the order of magnitude for $R_\Delta(C_{2k}, s)$ for $k\not\in \{2,3,5\}$ is out of reach with the current state of knowledge, as any improvement to the best-known exponents would yield a corresponding improvement in the best-known exponents for $\mathrm{ex}(n, C_{2k})$. One should be able to remove the logarithmic factor in the denominator of Theorem \ref{general cycle} but we could not see an easy way to do this. Probably the most interesting open question in the area of degree Ramsey numbers is whether or not $R_\Delta(G, s)$ is bounded by some function of $\Delta(G)$ and $s$. Horn, Milans, and R\"odl \cite{HornMilansRodl2014} showed that this is true for the family of closed blowups of trees. However, it is not clear that it should be true for general $G$.

\section*{Acknowledgments}
The author would like to thank Boris Bukh, Ross Kang, and Felix Lazebnik for helpful discussions.

\bibliographystyle{plain}
\bibliography{bib}

\end{document}